\documentclass[a4paper,12pt]{amsart}
 
\usepackage{amssymb, amsmath, amsthm, amsfonts}
 
\usepackage{hyperref}

\usepackage{booktabs}
\theoremstyle{plain}
\newtheorem{theorem}{Theorem}[section]

\newtheorem{proposition}[theorem]{Proposition}

\theoremstyle{definition}
\newtheorem{definition}[theorem]{Definition}
\newtheorem{remark}[theorem]{Remark}

\def\r{\mathbb R}
 
 \def\nil{\text{Nil}_3}
\usepackage{xcolor,graphicx}

\begin{document}
\title[Solitons in the Heisenberg space $\nil$]{Classification of invariant Gauss curvature   solitons in the Heisenberg space $\nil$}
\author{Rafael Belli}
\address{Department of Mathematics. Federal University of São Carlos. 13565-905 São Carlos, Brazil}
\email{rafaelbelli@estudante.ufscar.br}
\author{Rafael L\'opez}
\address{Department of Geometry and Topology. University of Granada. 18071 Granada, Spain}
\email{rcamino@ugr.es}
\maketitle

\begin{abstract}
In this paper, we classify all solitons of the Gauss curvature flow in the three-dimensional Heisenberg group $\mathrm{Nil}_3$ that are
 invariant under a one-parameter group of ambient isometries. By means of   the four canonical types of Killing vector 
 fields and   the three   families of invariant surfaces (vertical translations, horizontal translations, 
 and helicoidal motions), we analyze the twelve resulting types of possible solitons.   In some cases, there do not exist any invariant solitons; in others, we find explicit parametrizations, or   describe their geometric properties.
 
\end{abstract}

\section{Introduction and statement of the main results}

 The Gauss curvature flow (GCF) is a classical extrinsic geometric flow where a surface evolves along its unit normal vector field with a velocity proportional to its  Gauss curvature. While traditionally studied in the Euclidean setting, this flow can be naturally extended to non-Euclidean ambient spaces, such as the Heisenberg group $\nil$. To be precise, let $\psi:\Sigma\to\nil$ be a smooth immersion of a surface $\Sigma$ into the Heisenberg group $\nil$. We say that $\psi$ evolves by the Gauss curvature flow if there exists a smooth map $\Psi:(-\epsilon,\epsilon)\times \Sigma \to \nil$ such that, for every $t \in (-\epsilon,\epsilon)$, the map $\Psi(t,\cdot):\Sigma \to \nil$ is a smooth immersion satisfying the parabolic system
\begin{equation}\label{eq0}
\left\{ 
\begin{array}{l}
\frac{\partial \Psi}{\partial t}(t,\cdot) = -K(t,\cdot) N(t,\cdot),
\\ 
\Psi(0,\cdot)=\psi,
\end{array}%
\right.
\end{equation}
where $K(t,\cdot)$ and $N(t,\cdot)$ denote, respectively, the extrinsic Gauss curvature and the unit normal vector field associated with the immersion $\Psi(t,\cdot)$.

A particularly interesting class of solutions to the GCF consists of the self-similar solutions, commonly referred to as \emph{solitons}. The  shapes of these surfaces remain invariant under the flow, meaning they evolve purely by ambient isometries of $\nil$. Formally, let $\{\Phi_t\}_{t\in\mathbb{R}}$ be a one-parameter family of isometries of $\nil$ with $\Phi_0=\mathrm{Id}$. A smooth immersion $\psi:\Sigma\to\nil$ is called a GCF \emph{soliton} associated with $\{\Phi_t\}_{t\in\mathbb{R}}$ if  
\begin{equation}\label{eq00}
\Psi(t,p)=\Phi_t(\psi(p)), \qquad t\in \r, p\in\Sigma,
\end{equation}
is a solution of the GCF equation \eqref{eq0}. If $F$ is the Killing vector field in $\nil$ generated by $\{\Phi_t\}_{t\in\mathbb{R}}$, we can decompose it into its tangential and normal components along the surface. Since the tangential part merely induces time-dependent internal reparametrizations without altering the geometric image of the evolving profile, a standard argument allows us to rule out these tangential shifts and characterize the solitons \eqref{eq00} purely in terms of a stationary equation involving $K$ and $N$ \cite{ck,dt}. We adopt this geometric characterization as our working definition of a GCF soliton.

\begin{definition}
Let $F\in\mathcal{X}(\nil)$ be a Killing vector field. A surface   $\Sigma$ immersed in $\nil$ is said to be an {\it $F$-soliton} of the GCF   if 
\begin{equation}\label{eq1}
K=-\langle N,F\rangle,
\end{equation}
where $K$ and $N$ are the extrinsic Gauss curvature and the unit normal vector field of $\Sigma$, respectively.
\end{definition}

The Gauss curvature flow in Euclidean space has its starting point in the study by   Firey to model the shape of tumbling stones \cite{firey}. The  literature on the  GCF in Euclidean space is vast, with initial pioneering works including   \cite{an,ch,ts,ur}, without to be a complete list.  However, the extension of   the  GCF   to other  Riemannian 
manifolds  has been limited \cite{andrews94,chu,delima,limapipoli,wy,wyz}. Recently, the authors of the present paper have classified the GCF solitons in the solvable group \cite{bl1}, in the product space $\mathbb{H}^2\times\r$ \cite{bl2} and hyperbolic space \cite{lopez}. In contrast, research   on the mean curvature flow in   homogenous spaces has been much more active, including the eight geometries of Thurston: \cite{bu1,bu2,bu3,ch2,lima,lira,lm5,mm,pi,pi2}.

In this paper, we classify  the GCF solitons in the Heisenberg group $\nil$. Depending on the choice of the Killing vector field $F$ in \eqref{eq1}, we obtain different families of GCF solitons. To simplify this work, we recall that the isometry group  of $\nil$ is four-dimensional. A canonical basis $\{F_1,F_2,F_3,F_4\}$ of Killing vector fields is associated to   vertical translations, two types of horizontal translations, and   rotations around the vertical axis, respectively.

 Since the elliptic equation \eqref{eq1}  is  remarkably  difficult to solve in all its  generality, we restrict our study to solitons that exhibit   high symmetry properties with respect to the ambient space $\nil$. More precisely, we consider solitons that are invariant under a one-parameter group of isometries of $\nil$.
 
   \begin{definition} \label{df2}
 An immersed surface $\Sigma$ in $\nil$ is said to be invariant under the action of a one-parameter group of isometries $\{\Phi_t\colon t\in\r\}$ if $\Phi_t(\Sigma) = \Sigma$ for all $t$.
  \end{definition} 
 In the Heisenberg group, the one-parameter groups of isometries  fall into four     classes,  depending on the Killing vector field that generates each group:  vertical translations, horizontal translations (of two types),   and helicoidal motions. However, one type of horizontal translation is obtained from the other type after a rotation of $\nil$. This allows us to reduce the analysis to three types of invariant surfaces.
 
It is important to emphasize  that the Killing field   defining the soliton flow in \eqref{eq1} and the Killing field   defining the   invariance 
of the surface $\Sigma$ need not coincide. Therefore, in the classification of solitons invariant under a one-parameter group of isometries, we distinguish between the type of Killing vector field  of the soliton equation \eqref{eq1}   and the type of invariance symmetry characterizing the surface geometry  (Definition \ref{df2}).

The paper is organized as follows. In Section \ref{s2}, we recall the necessary   background on the Heisenberg group $\nil$, introduce its global frame of left-invariant vector fields, and present the canonical classification  for invariant surfaces. Sections \ref{s3}, \ref{s4}, \ref{s5}, and \ref{s6} are devoted to the   study and classification of invariant solitons   associated with  the Killing vector fields $F_1$, $F_2$, $F_3$, and $F_4$, respectively. In each case, equation \eqref{eq1}  is reduced to nonlinear ordinary differential equations, allowing us to find explicit solutions, obtain first integrals, or characterize the periodic and qualitative behavior of the corresponding generating curves.  Thus, since there are four basic Killing vector fields for the flow, the total number of  invariant solitons in $\nil$ is $12$. A summary of this classification is given in Table \ref{table1}. In some cases, there are no solitons; in others, we obtain explicit parametrizations, and in the rest, we can determine the main geometric properties of the solitons via the ODE \eqref{eq1}.

 \begin{table}[htpb]
\centering
\caption{Summary of invariant Gauss curvature  solitons in $\nil$.}\label{table1}
\begin{tabular}{@{}lccc@{}}
\toprule
 $F_k$-soliton & \begin{tabular}{@{}c@{}} Vertical \\ translation\end{tabular} & \begin{tabular}{@{}c@{}}Horizontal \\ translation\end{tabular} & \begin{tabular}{@{}c@{}}Helicoidal \\ motion\end{tabular} \\ \midrule
$F_1$ & \begin{tabular}{@{}c@{}}Explicit \\ vertical planes\end{tabular} & \begin{tabular}{@{}c@{}}Flat ($K=0$) \\ ODE \eqref{16}\end{tabular} & None \\ \addlinespace
$F_2$ & \begin{tabular}{@{}c@{}}Explicit \\ vertical planes\end{tabular} & None & None \\ \addlinespace
$F_3$ & None & \begin{tabular}{@{}c@{}}ODE \eqref{18} \\ explicit curve \eqref{17}\end{tabular} & ODE \eqref{23} \\ \addlinespace
$F_4$ & Explicit curves \eqref{f4v} & None & \begin{tabular}{@{}c@{}}ODE \eqref{f4r0}\\ rotational: ODE \eqref{f4r}  \end{tabular} \\ \bottomrule
\end{tabular}
\end{table}

\section{Preliminaries}\label{s2}

In this section, we find a canonical basis of Killing vector fields in the Heisenberg group and we recall the classification of invariant surfaces.

The Heisenberg group   is the three-dimensional Lie group $(\mathbb{R}^3,\ast)$, where the group product $\ast$ is defined, for any pair of points $(x,y,z), (x',y',z')\in\mathbb{R}^3$, by 
$$(x,y,z)\ast (x',y',z')=\left(x+x', y+y', z+z' + \frac{1}{2}(x y' - y x')\right).$$
This space is endowed with a family of  the left invariant metrics  
$$ds^2 = dx^2+dy^2+\left(\tau(y dx-x dy)+dz\right)^2,$$
where $\tau\not=0$ is a real parameter. Under these metrics, the Heisenberg group becomes  a homogeneous Riemannian manifold with a $4$-dimensional isometry group. Thus, the   Heisenberg group is   viewed as a member of the family of homogeneous spaces $\mathbb E(\kappa,\tau)$ 
  with $\kappa=0$ and $\tau\neq 0$ \cite{daniel,mp}. However,  all such spaces are homothetic. Indeed, if the metric is rescaled by a constant factor 
  $  \lambda^2$, then the bundle curvature transforms according to $\tau \mapsto \tau/\lambda$. Under a homothetic change of   metric $\bar g=\lambda^2 g$, the Gauss curvature and the unit normal transform as
$\bar K=\lambda^{-2}K$ and $\bar N=\lambda^{-1}N$. Hence, the soliton equation \eqref{eq1} is   preserved after rescaling the 
ambient Killing field according to $\bar F=\lambda^{-1}F$. Therefore, no   
  generality is lost by fixing a precise space $E(0,\tau)$. In this paper, we choose $\tau=1/2$ and we refer the Heisenberg group as $\nil=E(0,\frac12)$.

In $\nil$, a global orthonormal tangent frame $\mathcal{B}=\{E_1,E_2,E_3\}$ of left-invariant vector fields is given by  
\begin{equation*}
\begin{split}
E_1 &= \partial_x - \frac{y}{2} \partial_z, \\
 E_2 &= \partial_y + \frac{x}{2} \partial_z, \\
 E_3 &= \partial_z.
 \end{split}
\end{equation*}
The Riemannian connection $\bar{\nabla} $ in terms of $\mathcal{B}$ can be expressed with the matrix 
\begin{equation}\label{co}
(\bar{\nabla}_{E_i}E_j)=\frac12\left( \begin{array}{ccc}
 0  &    E_3  &   -  E_2 \\
  -  E_3  &  0 &     E_1  \\
  -  E_2 &    E_1  &   0\end{array}\right).
  \end{equation}
The four-dimensional Lie algebra of Killing vector fields is globally spanned by the basis $\{F_1,F_2,F_3,F_4\}$, which can be expressed in coordinates with respect to $\mathcal{B}$ as:
\begin{equation*}\label{killingbasis}
\begin{split}
F_1&=\partial_x+\frac{y}{2}\partial_z=E_1+y E_3,\\
F_2&=\partial_y-\frac{x}{2}\partial_z=E_2-x E_3,\\
F_3&=\partial_z=E_3,\\
F_4&=-y\partial_x+x\partial_y=-yE_1+x E_2-\frac{x^2+y^2}{2}E_3.
\end{split}
\end{equation*}
We will restrict the study of solitons to $F_k$-solitons, with $1\leq k\leq 4$. 

On the other hand, the classification of the one-parameter groups of isometries in $\nil$, and consequently, of invariant surfaces (Definition \ref{df2}), is obtained via  the basis $\{F_1,F_2,F_3,F_4\}$.  Invariant surfaces   were deeply studied and classified in \cite{fm}. The classification of invariant surfaces  with constant Gauss curvature is well-known:  invariant under translations   \cite{be,ino}, and  invariant by helicoidal and rotational motions    \cite{ak,caddeo,ik}.

Geometrically, $F_1$ and $F_2$ generate the horizontal translations, $F_3$ generates   vertical translations, and $F_4$ generates the rotational transformations around the vertical $z$-axis of $\r^3$. The explicit coordinate expressions for the one-parameter subgroups of isometries generated by the basic fields are, respectively, the following: \begin{equation}\label{gr}
 \begin{split}
F_1\colon \quad (x,y,z)&\mapsto (x+t,y,z+\frac{y}{2}t), \\
F_2\colon \quad (x,y,z)&\mapsto (x,y+t,z-\frac{x}{2}t),\\
F_3\colon \quad (x,y,z)&\mapsto (x,y,z+t),\\
F_4\colon \quad (x,y,z)&\mapsto (x\cos t -y\sin t,x\sin t+y\cos t,z).
\end{split}
\end{equation}
If the one-parameter subgroup is generated by a generic Killing vector field of the form $X=a_1F_1+a_2F_2+a_3F_3$, with $a_1^2+a_2^2\not=0$, then up to an isometry  of $\nil$, the subgroup is generated by a Killing vector field of the form $aF_1+cF_3$, with $a\not=0$. Indeed, if $R_t$ denotes the rotation about the $z$-axis given in \eqref{gr}, then $(R_t)_*F_1=\cos t\,F_1+\sin t\,F_2$.  Consequently, for any non-trivial linear combination $ a_1F_1+a_2F_2$, with $ (a_1,a_2)\neq(0,0)$,  choosing $t$ such that $a_1=\sqrt{a_1^2+a_2^2}\cos t$ and $ 
a_2=\sqrt{a_1^2+a_2^2}\sin t$, one obtains $X=\sqrt{a_1^2+a_2^2}\,(R_t)_*F_1$.

This geometric reduction  simplifies the classification of  invariant surfaces because,  up to ambient isometries, it suffices to consider the one-parameter groups given by $F_3$, $F_1$ and $ F_4+cF_3$,  where $c\in\r$:  
\begin{enumerate}
\item surfaces invariant under vertical translations (generated by $F_3$);
\item surfaces invariant under horizontal translations (generated by $F_1$);
\item surfaces invariant under   helicoidal motions (generated by $F_4+cF_3$). If $c=0$, the surface is invariant under   rotational motions (generated by $F_4$).
\end{enumerate}

The three types of invariant surfaces together the four types of solitons yields a total of 12 types of invariant solitons in $\nil$.

As a first step, we need suitable parametrizations of the invariant surfaces of $\nil$. For this purpose, it is necessary  to find a generating curve $\gamma$, which   must be transversal to the orbits of the   group of isometries.  Once we have  the curve $\gamma$, we apply the corresponding group  to obtain the parametrization of the surface.

\begin{proposition}\label{pr21}
Up to ambient isometries, invariant surfaces in $\nil$ can be parametrized as follows.
\begin{enumerate}
\item  Vertical translation invariant surfaces.  A generating curve is $\gamma(s)=(x(s),y(s),0)$, $s\in I$, and the  invariant surface is parametrized by
\begin{equation}\label{p1}
\Psi(s,t)=\bigl(x(s),y(s),t\bigr),\quad (s, t) \in I \times \r.
\end{equation}
\item  Horizontal translation invariant surfaces by $F_1$.  A generating curve is $\gamma(s)= (0,y(s),z(s) )$, $s\in I$, 
 and the invariant surface is   parametrized by  
 \begin{equation}\label{p2}
 \Psi(s,t)=\left(t,\,y(s),\,z(s)+\frac{y(s)}{2}t\right),\quad (s, t) \in I \times \r.
\end{equation}
\item Helicoidal surfaces. A generating curve  is $\gamma(s)= (r(s),0,h(s) )$, $s\in I$, and the invariant surface is parametrized by
 \begin{equation}\label{p3}
 \Psi(s,t)=\bigl(r(s)\cos t,\,r(s)\sin t,\,h(s)+ct\bigr),\quad (s, t) \in I \times \r.
\end{equation}

\end{enumerate}
\end{proposition}

Now we need the expressions for $K$ and $N$ that appear in the soliton equation \eqref{eq1}.

\begin{proposition}\label{pr22}
Let $\Sigma$ be an invariant surface in $\nil$ parametrized as in Proposition \ref{pr21}. The Gauss curvature $K$  and the unit normal vector (in coordinates with respect to $\mathcal{B}$) of $\Sigma$ are given by
\begin{enumerate}
\item Vertical translation invariant surfaces. Suppose $x'^2 + y'^2 = 1$. Then  
\begin{equation}\label{n1}
N =( y',-x',0),
\end{equation}
\begin{equation}\label{k1}
K =-\frac14.
\end{equation}
\item Horizontal translation invariant surfaces. Suppose $y'^2 + z'^2 = 1$. Then  
\begin{equation}\label{n2}
N =  \frac{1}{\sqrt{W}}(-yy', - z', y' ), \quad \text{where } W = \sqrt{1+ y^2y'^2},
\end{equation}
\begin{equation}\label{k2}
K =-\frac{4 y y''+y^4 y'^4+y^2 \left(2 y'^2-4 y'^4\right)+1}{4W^2}.\end{equation}
\item  Helicoidal surfaces. Suppose $r'^2 + h'^2 = 1$. Then 
\begin{equation}\label{n3}
N = \frac{1}{\sqrt{W}} \left( r'\left(c - \frac{r^2}{2}\right)\sin t - rh'\cos t, - \left(r'\left(c - \frac{r^2}{2}\right)\cos t + rh'\sin t\right) , rr' \right),
\end{equation}
where $W = r^2 + r'^2\left(c-\frac{r^2}{2}\right)^2$.  
\begin{equation}\label{k3}
\begin{split}
K &=  \frac{1}{W} \Bigg( \left( 1 - c + \frac{r^2}{2} \right) \left(-r^3r''+ r^2r'^2(1 - r'^2)\left( c - \frac{r^2}{2} \right)  \right) \\
&- \left[ r'^2 \left( c - \frac{r^2}{2} \right) \left( \frac{c}{2} - \frac{r^2}{4} - 1 \right) - \frac{r^2}{2} \right]^2 \Bigg).
\end{split}
\end{equation}
\end{enumerate}
\end{proposition}

\begin{proof}
For a parametrization $\Psi$ of a surface $\Sigma$, let  $g_{ij}=\langle \Psi_i, \Psi_j \rangle$  and $h_{ij} = \langle \bar{\nabla}_i \Psi_j, N \rangle$ denote the coefficients of the first and second fundamental form, respectively.  Let $W=g_{11}g_{22}-g_{12}^2$ be the determinant of the metric matrix $(g_{ij})$. The Gauss curvature is calculated via $K = \frac{h_{11}h_{22}-h_{12}^2}{W}$. For the determination of the unit normal vector $N$, we calculate the cross-product $\Psi_s\times \Psi_t$ and next, we divide by its modulus, which coincides with $\sqrt{W}$.

We now consider the three types of invariant surfaces parametrized by \eqref{p1}, \eqref{p2} and \eqref{p3}. In the computations, we use the connection matrix \eqref{co} together the expressions of the basic vector fields $\{\partial_x,\partial_y,\partial_z\}$ in terms of $\mathcal{B}$, namely, $\partial_x=E_1-\frac{y}{2}E_3$, $\partial_y=E_2+\frac{x}{2}E_3$ and $\partial_z=E_3$.

\begin{enumerate}
\item    The basic tangent vectors are $\Psi_s = x' \partial_x + y' \partial_y$ and $\Psi_t = \partial_z$, which expressed in terms of   $\mathcal{B}$ are:
$$\Psi_s = (x' , y' , \frac{1}{2}(x'y - y'x)), \quad \Psi_t = (0,0,1).$$
The metric coefficients are $g_{11} = 1 + \frac{1}{4}(x'y - y'x)^2$, $g_{12} = \frac{1}{2}(x'y - y'x)$, and $g_{22} = 1$, which gives $W= 1$. The unit normal vector is $N = \Psi_s \times \Psi_t =( y', - x',0)$.  

Using the connection matrix \eqref{co}, we have $\bar{\nabla}_t \Psi_t = 0$. Thus, for the calculation of $K$, it suffices  to compute
 $\bar{\nabla}_s \Psi_t  =\frac{1}{2}(y'E_1 -  x'E_2)$. Then    $h_{12}= - \frac{1}{2}$, and this gives $K=-1/4$.

\item The basic tangent vectors are  $\Psi_s =(0, y', z')$ and $\Psi_t = (1,0,y)$. The metric coefficients are $g_{11} = y'^2+z'^2 = 1$, $g_{12} =  yz' $, and $g_{22} = 1+y^2$. Then  $W = 1+y^2y'^2$. The cross-product of $\Psi_s$ and $\Psi_t$ gives  \eqref{n2}.  

For the coefficients $h_{ij}$ of the second fundamental form, we first find 
\begin{equation*}
\begin{split}
\bar{\nabla}_s \Psi_s &=  (y'z', y'', z''),\\
\bar{\nabla}_s \Psi_t &= (  \frac{1}{2}yy', - \frac{1}{2}z', \frac{1}{2}y'),\\
\bar{\nabla}_t \Psi_t &= (0, -y,0).
\end{split}
\end{equation*}
This gives 
$$K = \frac{4yz'(y'z'' - z'y'' - yy'^2z') - (1 - y^2y'^2)^2}{4W^2}.$$
Using $z'^2=1-y'^2$ and $z'z''=-y'y''$, we obtain the desired value for $K$ given in \eqref{k2}.

\item The basic  tangent vectors of the surface are 
\[
\Psi_s =( r'\cos t  , r'\sin t , h'), \quad \Psi_t = (-r\sin t , r\cos t , (c - \frac{r^2}{2})).
\]
The coefficients $g_{ij}$   are:
$$g_{11} = 1, \quad g_{12} = h'\left(c - \frac{r^2}{2}\right), \quad g_{22} = r^2 + \left(c - \frac{r^2}{2}\right)^2.$$
Using that $r'^2+h'^2=1$, we have $ W = r'^2 \left( c - \frac{r^2}{2} \right)^2 + r^2$. 
For  the coefficients of the second fundamental form, and in coordinates with respect to $\mathcal{B}$, we have 
\begin{equation*}
\begin{split}
\bar{\nabla}_{\Psi_s}\Psi_s&=(h' r'\sin t +r''\cos t  , r''\sin t- h' r'\cos t,h''),\\
\bar{\nabla}_{\Psi_s}\Psi_t&=\frac14\left(  \left(2 r h'\cos t  -r' \left(-2 c+r^2+4\right) \sin t\right),  \left(r' \left(-2 c+r^2+4\right) \cos t+2 r \sin t h'\right),-2 r r' \right),\\
\bar{\nabla}_{\Psi_t}\Psi_t&=\frac12(r \cos t  (2c-  r^2-2 ),   r \sin t \left( 2 c-r^2-2\right),0).
\end{split}
\end{equation*}
This yields
\begin{equation*}
\begin{split}
h_{11} &= \frac{1}{\sqrt{W}} \left[ r(r'h'' - r''h') + h'r'^2 \left( c - \frac{r^2}{2} \right) \right],\\
 h_{12} &= \frac{-4 r^2 \left((c-1) r'^2+1\right)+4 (c-2) c r'^2+r^4 r'^2}{8 \sqrt{W}},\\
 h_{22} &= -\frac{r^2 h'}{\sqrt{W}} \left(c - 1 - \frac{r^2}{2}\right).
\end{split}
\end{equation*}
This gives the expression for $K$  given in \eqref{k3}, after using the relations $h'^2=1-r'^2$ and $h'h''=-r'r''$.
\end{enumerate}
\end{proof}

  We have stated that there are 12 types of invariant solitons in $\nil$ because we combine the four possible generating fields  $\{F_1, F_2, F_3, F_4\}$ for the 
  soliton equation \eqref{eq1}with the three canonical geometries for the invariant surfaces (horizontal translations, 
  vertical translations, and helicoidal motions). We point out that the  study of $F_2$-solitons can be reduced to $F_1$-solitons 
  as is proven in the following result. 

\begin{proposition}
\label{pr23}
  Consider the isometry $\phi: \text{Nil}_3 \to \text{Nil}_3$ given by
$$\phi(x, y, z) = (-y, x, z).$$
If $\Sigma \subset \text{Nil}_3$ is an immersed surface, then $\Sigma$ is an $F_1$-soliton  if and only if the image surface $\bar{\Sigma} = \phi(\Sigma)$ is an $F_2$-soliton.
\end{proposition}

\begin{proof}  
Note that $\phi$ is the $\pi/2$-rotation that appeared in \eqref{gr}. 
 Let $K$ and $\bar{K}$ denote the respective Gauss curvatures of $\Sigma$ and $\bar{\Sigma}$, and let $N$ and $\bar{N}$ be their unit normal fields. Because $\phi$ is an isometry, we have   $\bar{K}(\phi(p)) = K(p)$, and  $\bar{N}(\phi(p)) = \phi_*(N(p))$, where   $\phi_*$ is the pushforward of $\phi$. 

Let $\phi(x,y,z)=(u(x,y,z),v(x,y,z),w(x,y,z))$, where $u = -y$, $v = x$ and $w=z$. Recall that    $F_1 = \partial_x + \frac{y}{2}\partial_z$ and $F_2=\partial_y-\frac{x}{2}\partial_z$. Then
\begin{equation*}
\phi_*(F_1)=\phi_*(\partial_x)+\frac{y\circ \phi^{-1}}{2}\phi_*(\partial_z) =\partial_v-\frac{u}{2}\partial_w=F_2\circ\phi.
\end{equation*}

Assume that $\Sigma$ is a soliton with respect to $F_1$. Then  $K(p) = -\langle N(p), F_1|_p \rangle$ for all $p \in \Sigma$. Then, evaluating on the image surface $\bar{\Sigma}$, we have 
\begin{equation*}
\begin{split}
\bar{K}(\phi(p)) &= K(p) = -\langle N(p), F_1|_p \rangle = -\langle \phi_*(N(p)), \phi_*(F_1|_p) \rangle = -\langle \bar{N}(\phi(p)), F_2|_{\phi(p)} \rangle\\ 
& = -\langle \bar{N}, F_2 \rangle\circ\phi,
\end{split}
\end{equation*}
 which implies that $\bar{\Sigma}$ is an $F_2$-soliton.
 
The converse follows immediately by reversing the sequence of equalities, since $\phi$ is a global isometry of $\nil$ and its inverse $\phi^{-1}(u,v,w) = (v,-u,w)$ satisfies $(\phi^{-1})_*(F_2) = F_1 \circ \phi^{-1}$.
\end{proof}

However, this proposition cannot be used to further reduce the total number of 12 invariant soliton cases. The   reason is that the isometry $\phi$  applied to change $F_2$-soliton into $F_1$-soliton also acts   on the surface, thereby  transforming its geometric symmetry. 
Indeed, we have already utilized rotations of $\nil$  to normalize the one-parameter groups of isometries into three canonical forms.   If we 
apply an additional rotation to transform, for instance, an $F_2$-soliton into an $F_1$-soliton, the same rotation will alter the   invariance of the surface.   Consequently, the combinations of solitons   and invariant surfaces     must be 
analyzed as   12 distinct  cases.  

\section{Invariant solitons with respect to the vector field $F_1=\partial_x+\frac{y}{2}\partial_z$}\label{s3}

In this section, we investigate the first type of solitons, the $F_1$-solitons, where  $F_1=\partial_x+\frac{y}{2}\partial_z$. In the following result, we obtain a first classification of the invariant $F_1$-solitons, where equation \eqref{eq1} is evaluated in each of the three types of invariant surfaces.

\begin{proposition}
An invariant surface $\Sigma$ in $\text{Nil}_3$ is an $F_1$-soliton if and only if one of the following conditions holds:
\begin{enumerate}
\item $\Sigma$ is invariant under vertical translations, in which case it is a vertical plane parametrized by
\begin{equation}\label{f1h}
\Psi(s, t) = \left( \pm \frac{\sqrt{15}}{4}s + x_0,    \frac{1}{4}s + y_0,  t \right), \quad x_0, y_0 \in \mathbb{R}.
\end{equation}
\item $\Sigma$ is invariant under horizontal translations, in which case it is flat ($K=0$), and its generating curve satisfies the ODE\begin{equation}\label{16}
4 y y''+y^4 y'^4+y^2 \left(2 y'^2-4 y'^4\right)+1 = 0.
\end{equation}
\item There are no $F_1$-solitons invariant under helicoidal motions.
\end{enumerate}
\end{proposition}

\begin{proof}
In terms of the left-invariant orthonormal frame $\mathcal{B}$, the Killing vector field $F_1$ is expressed as $F_1 = E_1 + y E_3$. Thus, if   $N = (N_1, N_2, N_3)$ denotes the unit normal vector field, then  $\langle N, F_1 \rangle = N_1 + y N_3$. 
\begin{enumerate}
    \item For vertical translation invariant surfaces parametrized by \eqref{p1}, with $x'^2+y'^2=1$,  equation \eqref{n1} yields $\langle N, F_1 \rangle = y'$. Since $K = -1/4$, the soliton equation   becomes $y' = 1/4$. Substituting  this into  the arc-length condition $x'^2 + y'^2 = 1$, we obtain $x'^2 =  \frac{15}{16}$, which yields $x(s) = \pm \frac{\sqrt{15}}{4}s + x_0$.      
    \item For horizontal translation invariant surfaces parametrized by  \eqref{p2}, with $y'^2+z'^2=1$, the unit normal is given by  $N = \frac{1}{\sqrt{W}}(-yy', -z', y')$, where $W=(1+y^2)y'^2+z'^2$, which implies  $\langle N, F_1 \rangle = 0$. Thus, the soliton equation \eqref{eq1} reduces to $K = 0$. Using \eqref{k2},  we arrive at \eqref{16}.
    \item By virtue of  \eqref{n3} and \eqref{k3}, the soliton equation \eqref{eq1} reads
     \begin{equation*}
        \begin{split}
        &  \left( 1 - c + \frac{r^2}{2} \right)\left(-r^3r'' + r^2r'^2(1 - r'^2)\left( c - \frac{r^2}{2} \right)  \right) \\
        &\quad - \left[ r'^2 \left( c - \frac{r^2}{2} \right) \left( \frac{c}{2} - \frac{r^2}{4} - 1 \right) - \frac{r^2}{2} \right]^2   = - \sqrt{W} \left( r'\left(c + \frac{r^2}{2}\right)\sin t - rh'\cos t \right),
        \end{split}
    \end{equation*}
    where $W = r^2 + r'^2\left(c-\frac{r^2}{2}\right)^2$.  Since the  left-hand side is independent of $t$,  the right-hand side implies that   $h'=0$ and $r'(c-\frac{r^2}{2})=0$. For $r$, we deduce that $r$ is constant and then, $r'=0$. This   contradicts the fact that   the generating curve is parametrized by arc-length. This proves that not such a surface exists. 
\end{enumerate}
\end{proof}

In the following result, we study the solutions of Eq. \eqref{16}. First, we obtain a first integral of this equation and next, we describe the   properties of the generating curve: see Fig. \ref{fig1}, left. 

\begin{theorem}\label{t32}
 A first integral of \eqref{16} is given by
\begin{equation}\label{i1}
\frac{1}{1+y^2y'^2}=1+y^2\!\left(C+\frac12\log |y|\right),
\end{equation}
where $C$ is an arbitrary constant. Moreover,  $y(s)$ is determined implicitly by  
\begin{equation}\label{i2}
s-s_0=\pm\int_{y_0}^{y}\frac{y\,dy}{\sqrt{\dfrac{1}{1+y^2\left(C+\frac12\log |y|\right)}-1}}.
\end{equation}
Let  $\gamma(s)=(y(s),z(s))$ be the generating curve contained in the $yz$-plane. Applying the symmetry $(y,z)\mapsto (-y,z)$ if necessary, we can assume $y>0$. Then     $\gamma$  has the following  properties:
\begin{enumerate}
\item The curve $\gamma$ is symmetric about a horizontal line of the $yz$-plane.
\item The function $y=y(s)$ has a unique extremum, which is a maximum. If $C\to \infty$, then $\gamma$ approaches the $z$-axis.
\item There exist constants $0<y_{\min}<y_{\max}$ such that $y_{\min}< y(s)\leq y_{\max}$ for all $s\in I$.   
\item The curve $\gamma$ is a concave graph over a bounded interval $(z_0,z_1)$ of the $z$-axis, and the tangent vectors of $\gamma$ at these points are parallel to the $y$-axis. \end{enumerate}

\end{theorem}

\begin{proof}
The function $y(s)$ cannot be zero; in fact, y(s) cannot approach $0$, that is, $y(s)$ remains strictly bounded away from $0$. This is because if $y\to 0$ in \eqref{16}, we arrive at the contradiction $0=1$. 

Let $p(y)=y'(s)$. Then $y'' = pp'$. Rewriting \eqref{16}, we obtain  
$$4y pp' + \bigl(1+y^2p^2\bigr)^2-4y^2p^4 = 0.$$
Introduce the variable $u(y)=y^2y'^2=y^2p^2$.  Differentiating with respect to $y$, 
\[u'=2yp^2+2y^2pp'.\]
 Substituting into the previous  equation, and using the fact that $p^2=u/y^2$, we have
$$2yu'-4u-4u^2+(1+u)^2y^2=0. $$
By setting $v=u+1$, this equation becomes
\begin{equation}\label{2y}
2y\,v'+v (4+(y^2-4)v)=0.
\end{equation}
Next, introduce
$$w=\frac1v.$$
Using $v'=-\frac{w'}{w^2}$,   equation \eqref{2y} transforms into the linear equation  
\[w'-\frac{2}{y}w=\frac{y}{2}-\frac{2}{y}.\]
The integrating factor is $\mu(y)=e^{\int -2/y\,dy}=y^{-2}$. 
Hence
\[\left(\frac{w}{y^2}\right)'=\frac1{2y}-\frac2{y^3}.\]
Integrating, and using that $y>0$, 
\[\frac{w}{y^2}=\frac12\log y+\frac1{y^2}+C,\]
and therefore
\[w=1+y^2\!\left(C+\frac12\log y\right).\]
Coming back to $y$, via $v$ and $u$, we arrive at    \eqref{i1}. The expression \eqref{i2} is immediate from isolating $y'$ in \eqref{i1}.

We proceed to prove the properties of $\gamma$. Since $\gamma$ is   parametrized by arc-length, there is a smooth function $\theta=\theta(s)$ such that   $y'=\cos\theta$ and $z'=\sin\theta$. The derivative $\theta'(s)$ represents the Euclidean curvature $\kappa(s)$ of  $\gamma(s)=(y(s),z(s))$. Then equation \eqref{16} writes as 
$$(4y\sin\theta)\theta'= 1+2y^2\cos^2\theta+(y^2-4)y^2\cos^4\theta.$$
Thus $\theta\in (0,\pi)$ and $\gamma$ is characterized by the system
\begin{equation}\label{d1}
\begin{split}
y'&=\cos\theta,\\
z'&=\sin\theta,\\ 
\theta'&= \frac{1+2y^2\cos^2\theta+(y^4-4y^2)\cos^4\theta}{4y\sin\theta}.
\end{split}
\end{equation}

Since the function $z$ does not appear in \eqref{16}, any vertical translation in $\nil$ preserves the solutions of \eqref{16}. Without loss of generality, we can assume initial conditions $y(0)=y_0>0$, $z(0)=0$ and $\theta(0)=\pi/2$. It is easy to check that the functions 
 $\bar{y}(s)=y(-s)$, $\bar{z}(s)=-z(-s)$ and $\bar{\theta}(s)=\pi-\theta(-s)$ are also   solutions of \eqref{d1} with the same initial conditions. By uniqueness, $\bar{\gamma}(s)=(\bar{y}(s),\bar{z}(s))=\gamma(s)$. This proves that $\gamma$ is symmetric about the line $z=0$. This proves   item (1).
 
 Let 
 $$f(y) = C + \frac{1}{2}\log y$$
 denote the parenthesis in the right-hand side of \eqref{i1}. Then, we have
\begin{equation}\label{et}
y'^2 = \frac{-f(y)}{1 + y^2 f(y)}.
\end{equation}
Since $0\leq y'^2\leq 1$, we have a necessary condition on the right-hand side of \eqref{et}.  
\begin{enumerate}
\item First, from   $y'^2\geq 0$,   we must have   $f(y) \le 0$. This yields $ C + \frac{1}{2}\log y \le 0$, or equivalently,  $ y \le e^{-2C} =: y_{\max}$. In particular, the function $y(s)$ is bounded from above. At   $y = y_{\max}$, we have $f(y_{\max}) = 0$, which implies $y' = 0$. From \eqref{16}, at this critical point, we have $y''=-\frac{1}{4y}<0$ because $y>0$, proving that $y$ reaches a local extremum.  In fact, any other critical point of $y$ also satisfies $y''=-\frac{1}{4y}<0$. Thus, all critical points should be local maximum. This proves the uniqueness of the maximum. 

It is immediate that if   $C \to +\infty$,   then $y_{\max}\to 0$ and $\gamma$ tends to the $z$-axis.  This proves   item (2).

\item We now study the condition  $y'^2 \leq 1$. This implies   
\begin{equation*}
   C + \frac{1}{2}\log y + \frac{1}{1+y^2} \ge 0.
\end{equation*}
Let us define the continuous  function $g(y) = C + \frac{1}{2}\log y + \frac{1}{1+y^2}$ on the interval $(0, y_{\max}]$. The function $g$ has the following properties:  
\begin{enumerate}
    \item    $\lim_{y \to 0^+} g(y) = -\infty$. Thus, the profile curve can never reach or cross the axis $y=0$. Moreover,   $g(y_{\max}) = \frac{1}{1+e^{-4C}} > 0$.
   \item The function $g(y)$ is strictly monotonically increasing because  $ g'(y) =   \frac{(1-y^2)^2}{2y(1+y^2)^2} \ge 0$. 
   \end{enumerate}  
    By the Intermediate Value Theorem, there exists a unique  $y_{\min} \in (0, y_{\max})$ such that $g(y_{\min}) = 0$. At this minimum height $y = y_{\min}>0$, we have $y'^2 = 1$ and $y_{\min}\leq y(s)$.  
 \end{enumerate}
 
 We now prove that $\gamma$ is a graph over a bounded interval $(z_0, z_1)$ of the $z$-axis and cannot be smoothly extended beyond $y = y_{\min}$. We have
\begin{equation}\label{ez}
z'^2 =   \frac{1+y^2}{1 + y^2 f(y)} g(y).
\end{equation}
For any $y \in (y_{\min}, y_{\max}]$, we have $g(y) > 0$ and $1+y^2 f(y) > 0$, which strictly implies $z'^2 > 0$. Hence, $z'(s)$ never vanishes in the interior of the profile, ensuring that $z(s)$ is strictly monotonic. Consequently, the curve   can be globally parametrized as a graph $y = y(z)$.

  Using the chain rule,  we obtain from \eqref{et} and \eqref{ez} that
$$\left(\frac{dz}{dy}\right)^2 = \frac{z'^2}{y'^2} = \frac{(1+y^2)g(y)}{-f(y)}.$$
Then, using that $f\leq 0$, we have
\begin{equation*}
\begin{split}
  z_{\max}-z_{\min} &= \int_{y_{\min}}^{y_{\max}} \sqrt{\frac{(1+y^2)g(y)}{-f(y)}} \, dy=\int_{y_{\min}}^{y_{\max}} \sqrt{\frac{ (1+y^2)f(y) + 1}{-f(y)}} \, dy\\
  &\leq \int_{y_{\min}}^{y_{\max}} \sqrt{\frac{ 1}{-f(y)}} \, dy.
  \end{split}
\end{equation*}
Since $f(y_{\max})=0$, we can write
$$-f(y) = f(y_{\max}) - f(y) = \frac{1}{2}\log\left(\frac{y_{\max}}{y}\right)\geq \frac{y_{\max} - y}{2y_{\max}},$$
where is the last inequality, we apply that  $\log(x) \ge 1 - \frac{1}{x}$ if $x \ge 1$. Then
$$z_{\max}-z_{\min}  \leq   \int_{y_{\min}}^{y_{\max}} \frac{\sqrt{2y_{\max}}}{\sqrt{y_{\max} - y}} \, dy = 2\sqrt{2y_{\max}(y_{\max} - y_{\min})}<\infty.$$ 
This establishes that the domain of $z$ is bounded. Moreover,  as $z$ approaches the endpoints $z_0$ and $z_1$, 
the value of $y$ approaches $y_{\min}$. 

To prove the concavity of the graph $y = y(z)$, we compute its second derivative. Using the chain rule, we have 
$dy/dz = y'/z' = \cot\theta$. Then,
$$\frac{d^2y}{dz^2} = \frac{\frac{d}{ds}(\cot\theta)}{z'} =   -\frac{\theta'}{\sin^3\theta}.$$
Since $\theta \in (0,\pi)$, we have   $d^2y/dz^2 < 0$ if and only if   $\theta' > 0$. From the system \eqref{d1}, the numerator of $\theta'$ can be rewritten by completing 
 the square as follows:
 $$1+2y^2\cos^2\theta+(y^4-4y^2)\cos^4\theta = (1-y^2\cos^2\theta)^2 + 4y^2\cos^2\theta\sin^2\theta.$$
 Since $y > 0$ and $\sin\theta > 0$, it 
 follows that $\theta' > 0$.   This implies that $\gamma$ is strictly concave. This concludes the proof of items (3) and (4).

\end{proof}
 
\section{Invariant solitons with respect to the vector field $F_2=\partial_y-\frac{x}{2}\partial_z$}\label{s4}

In this section, we study  the second   class of solitons, namely those Killing vector field is   $F_2=\partial_y-\frac{x}{2}\partial_z$. Following the   approach established in the previous section, we analyze the   three canonical families of invariant surfaces. As we 
will demonstrate,   only vertical planes as valid soliton solutions, thus ruling out horizontal and helicoidal invariant $F_2$-solitons.

\begin{proposition}
An invariant surface $\Sigma$ in $\text{Nil}_3$ is an $F_2$-soliton if and only if one of the following conditions 
holds:
\begin{enumerate}
\item $\Sigma$ is invariant under vertical translations, in which case it is a vertical plane parametrized 
by
\begin{equation}\label{f2h}
\Psi(s, t) = \left( -\frac{1}{4}s + x_0, , \pm \frac{\sqrt{15}}{4}s + y_0, t \right), \quad x_0, y_0 \in \mathbb{R}.
\end{equation}
\item There are no $F_2$-solitons invariant under horizontal translations.
\item There are no $F_2$-solitons invariant 
under helicoidal motions.
\end{enumerate}
\end{proposition}

\begin{proof}
With respect to the   left-invariant orthonormal frame $\mathcal{B}$, the Killing vector field $F_2$ is expressed as   $F_2 = E_2 - x E_3$. Thus, if $N = (N_1, N_2, N_3)$ denotes the unit normal vector of $\Sigma$ in this frame,  its inner product with $F_2$ is given by $\langle N, F_2 \rangle = N_2 - x N_3$. 

\begin{enumerate}
    \item For vertical translation invariant surfaces parametrized by \eqref{p1}, with $x'^2+y'^2=1$, equation \eqref{n1} yields 
     $\langle N, F_2 \rangle = -x'$. Since $K = -1/4$, the soliton equation \eqref{eq1} becomes  $x' = -1/4$. Substituting this into the arc-length condition $x'^2 + y'^2 = 1$, we obtain $y'^2 = \frac{15}{16}$, which yields $y(s) = \pm \frac{\sqrt{15}}{4}s + y_0$. This provides the desired parametrization.

    \item For horizontal translation invariant surfaces parametrized by \eqref{p2}, with $y'^2+z'^2=1$, we have $ \langle N, F_2 \rangle = \frac{-z' - ty'}{\sqrt{W}}$. Since the left-hand side of \eqref{eq1} depends solely on $s$,   the right-hand side   must be independent of $t$. This forces  $y' = 0$, and thus $y(s) = y_0$ is a constant function. From the arc-length condition $y'^2 + z'^2 = 1$, we deduce   $z' = \pm 1$.    Substituting this into \eqref{n2} and \eqref{k2},  we find $W=1$, $K=-1/4$, and $   \langle N, F_2 \rangle = \pm 1$. This yields a contradiction in the soliton equation \eqref{eq1}, proving  that  such solitons do not exist.

    \item For helicoidal surfaces parametrized by \eqref{p3}, with $y'^2+z'^2=1$, equation \eqref{n3} gives
    $$        \langle N, F_2 \rangle  =   -\frac{1}{\sqrt{W}} \left( r'\left(c + \frac{r^2}{2}\right)\cos t + rh'\sin t \right).$$
    Since the left-hand side of the soliton equation \eqref{eq1} depends only on $s$, we deduce $r'(c+\frac{r^2}{2})=h'=0$.  Therefore,  $h(s)=h_0$ is a constant.  The  arc-length condition $r'^2 + h'^2 = 1$ then implies that $r'^2 = 1$. However, since $ r'\left(c + \frac{r^2}{2}\right) = 0$, we must have $c + \frac{r^2}{2} = 0$. This means that $r(s)$ must be a constant function, so $r' = 0$, which   contradicts $r'^2 = 1$. Consequently, no helicoidal $F_2$-solitons exist.
\end{enumerate}
\end{proof}

\begin{remark}
We point out that the soliton \eqref{f2h} is the soliton   \eqref{f1h} via the map $\phi(x,y,z)=(-y,x,z)$. This is in accordance with the argument of Section \ref{s2} where the $\pi/2$-rotations map invariant surfaces under $F_2$-translations into invariant surfaces under $F_1$-translations, and with Proposition \ref{pr23}.
\end{remark}

\section{Invariant solitons with respect to the vector field $F_3=\partial_z$}\label{s5}

In this section, we investigate $F_3$-solitons.   Unlike the previous Killing vector fields $F_1$ and $F_2$ that imposed   geometric rigidity, the Killing vector field $F_3$  provides a rich family of solutions. We will show that while vertical translation solitons are excluded, both horizontal and helicoidal invariant surfaces produce   $F_3$-solitons characterized by   non-linear ordinary differential equations. In the following result, we assume that the invariant surfaces are parametrized as in Proposition \ref{pr21} under the conditions of Proposition \ref{pr22}.

 \begin{proposition}
An invariant surface $\Sigma$ in $\text{Nil}_3$ is an $F_3$-soliton if and only if one of the following conditions holds:
\begin{enumerate}
\item There are no $F_3$-solitons invariant under vertical translations.
    \item A horizontal translation invariant surface is an $F_3$-soliton if and only if $y(s)$ satisfies the ordinary differential equation:
    \begin{equation}\label{18}
       4 y y''+y^4 y'^4+y^2 \left(2 y'^2-4 y'^4\right)+1 = 4 y' W^{3/2},
    \end{equation}
    where $W = 1+ y^2y'^2$.
    \item A helicoidal  surface is an $F_3$-soliton if and only if $r(s)$ satisfies the ordinary differential equation:
    \begin{equation}\label{23}
        \begin{split}
        &  \left( 1 - c + \frac{r^2}{2} \right) \left(-r^3r''+ r^2r'^2(1 - r'^2)\left( c - \frac{r^2}{2} \right)  \right) \\
        &\quad - \left[ r'^2 \left( c - \frac{r^2}{2} \right) \left( \frac{c}{2} - \frac{r^2}{4} - 1 \right) - \frac{r^2}{2} \right]^2   = - rr' \sqrt{W} ,
        \end{split}
    \end{equation}
    where $W =  r^2 + r'^2\left(c-\frac{r^2}{2}\right)^2$.
    
    If the surface is rotational ($c=0$), the equation becomes
    \begin{equation*}
        \left( 1 + \frac{r^2}{2} \right)\left( r^3r'' + \frac{1}{2}r^4r'^2(1 - r'^2)  \right)    + \frac{r^4}{4}\left[ r'^2 \left( 1 + \frac{r^2}{4} \right) - 1 \right]^2   =  rr' \sqrt{W},
 \end{equation*}
    where $W  =r^2 + \frac{1}{4}r^4r'^2$.

\end{enumerate}
\end{proposition}

\begin{proof}
We evaluate the soliton equation \eqref{eq1} on the three families of invariant surfaces parametrized in  Proposition \ref{pr21}.  Note that with respect to $\mathcal{B}$, the Killing vector field is simply $F_3=(0,0,1)$, which implies that $\langle N,F_3\rangle=N_3$.  
\begin{enumerate}
    \item For vertical translation invariant surfaces, the normal   vector \eqref{n1}   yields $\langle N, E_3 \rangle = 0$. Since  $K = -1/4$, this yields a contradiction in \eqref{eq1}, meaning  no such solitons exist.
    \item For horizontal translation invariant surfaces, the third component of $N$ in \eqref{n2} is   $y'/\sqrt{W}$. Substituting this in \eqref{eq1}   yields    equation \eqref{18}
    \item  Equation \eqref{23} is straightforward from \eqref{n3} and \eqref{k3}.  
\end{enumerate}
\end{proof}
 
 In the following result, we give the main properties of the generating curves of invariant $F_3$-solitons under horizontal translations described in equation \eqref{18}.
 
\begin{theorem}
The generating curve $\gamma(s)=(y(s),z(s))$ of a horizontal invariant $F_3$-soliton  falls into one of the following two   cases:
\begin{enumerate}
    \item The profile curve can be expressed explicitly as a graph $z = z(y)$ given by
    \begin{equation}\label{17}
        z(y) = \pm \left( \frac{y}{2}\sqrt{15-y^2} + \frac{15}{2}\arcsin\left(\frac{y}{\sqrt{15-y^2}}\right) \right) + z_0,
    \end{equation}
    where $z_0$ is a constant of integration. See Fig. \ref{fig1}, right.
    
    \item  The function $y(s)$, where $s$ is the arc-length parameter,   satisfies the implicit first integral
    \begin{equation}\label{19}
        \frac{y'}{\sqrt{1+y^2y'^2}} + \frac{1}{4}\log  \left| \frac{4y'}{\sqrt{1+y^2y'^2}} - 1 \right| = \log  |y| + C, 
    \end{equation}
    where $C$ is an arbitrary constant of integration.
\end{enumerate}
\end{theorem}

\begin{proof}
Consider the solution of the ODE \eqref{18}. We make the same changes of variables as in the proof of Theorem \ref{t32}, obtaining a differential equation for $w$:
\begin{equation}\label{w}
 w' - \frac{2}{y}w = \frac{y}{2} - \frac{2}{y} - 2\varepsilon\sqrt{1-w},
 \end{equation}
where $\varepsilon=\mbox{sgn}(y')$. To solve this equation, and inspired by the   variable   $w/y^2$ that appeared in the proof of Theorem \ref{t32}, let us introduce 
$$P(y) = \frac{p}{\sqrt{1+y^2p^2}}.$$
Then  $w = 1 - y^2P^2$. Differentiating with respect to $y$ yields $w' = -2yP^2 - 2y^2P\frac{dP}{dy}$. Furthermore, assuming $y>0$,  equation \eqref{w} becomes
$$ -2yP^2 - 2y^2P\frac{dP}{dy} - \frac{2}{y}\left(1 - y^2P^2\right) = \frac{y}{2} - \frac{2}{y} - 2yP, $$
or equivalently,  
\begin{equation}\label{p4}
 4yP\frac{dP}{dy} = 4P-1.
 \end{equation}
\begin{enumerate}
\item Case $P=1/4$.  This implies 
 $$ y' = \pm\frac{1}{\sqrt{16-y^2}},$$
  Since the curve is parametrized by arc-length, we have  
\begin{equation*}
    z' = \pm \sqrt{\frac{15-y^2}{16-y^2}}.
\end{equation*}
Instead of solving this system of ODEs, it is more convenient to regard $z$ as a function of $y$, $z=z(y)$. Then 
$$z'(y)=\frac{z'(s)}{y'(s)}=  \pm \sqrt{15-y^2}.$$
The solution of this ODE  is \eqref{17}.
\item Case $P\not=1/4$. By separation of variables in \eqref{p4}, we have 
\begin{equation}\label{pp}
\frac{4P}{4P-1}\, dp=\frac{1}{y}\, dy.
\end{equation}
Integrating, we arrive at 
$$ P + \frac{1}{4}\log |4P-1| = \log |y| + C, $$
where $C$ is the integration constant. Replacing $P$ with its   definition in terms of $y'$ recovers \eqref{19}. 
\end{enumerate}
\end{proof}

\begin{remark} The integral equation \eqref{19} can be expressed in terms of the Lambert function $W$. This function is defined as a branch of the inverse of the complex function $\sigma\mapsto \sigma e^\sigma$ \cite{wiki}. If $W=W(w)$, then $W$ satisfies the ODE $w(1+W)W'=W$. This can be obtained from \eqref{pp} after the change of variables $w=y^4$, and $W=4P-1$.  
\end{remark}

  \begin{figure}[h!t]
\centering
\includegraphics[width=0.18\linewidth]{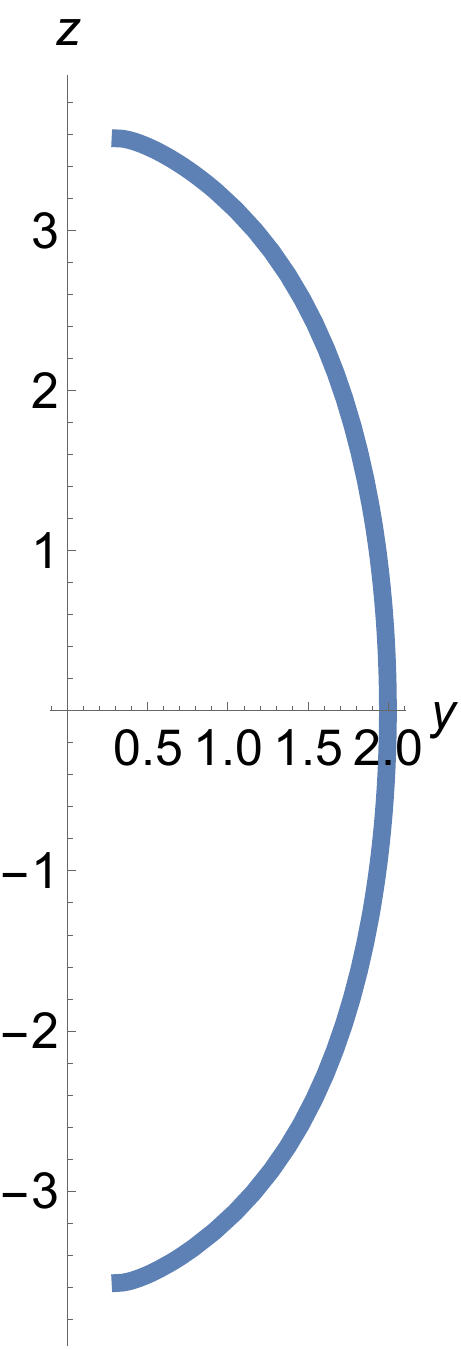}\qquad \includegraphics[width=0.45\linewidth]{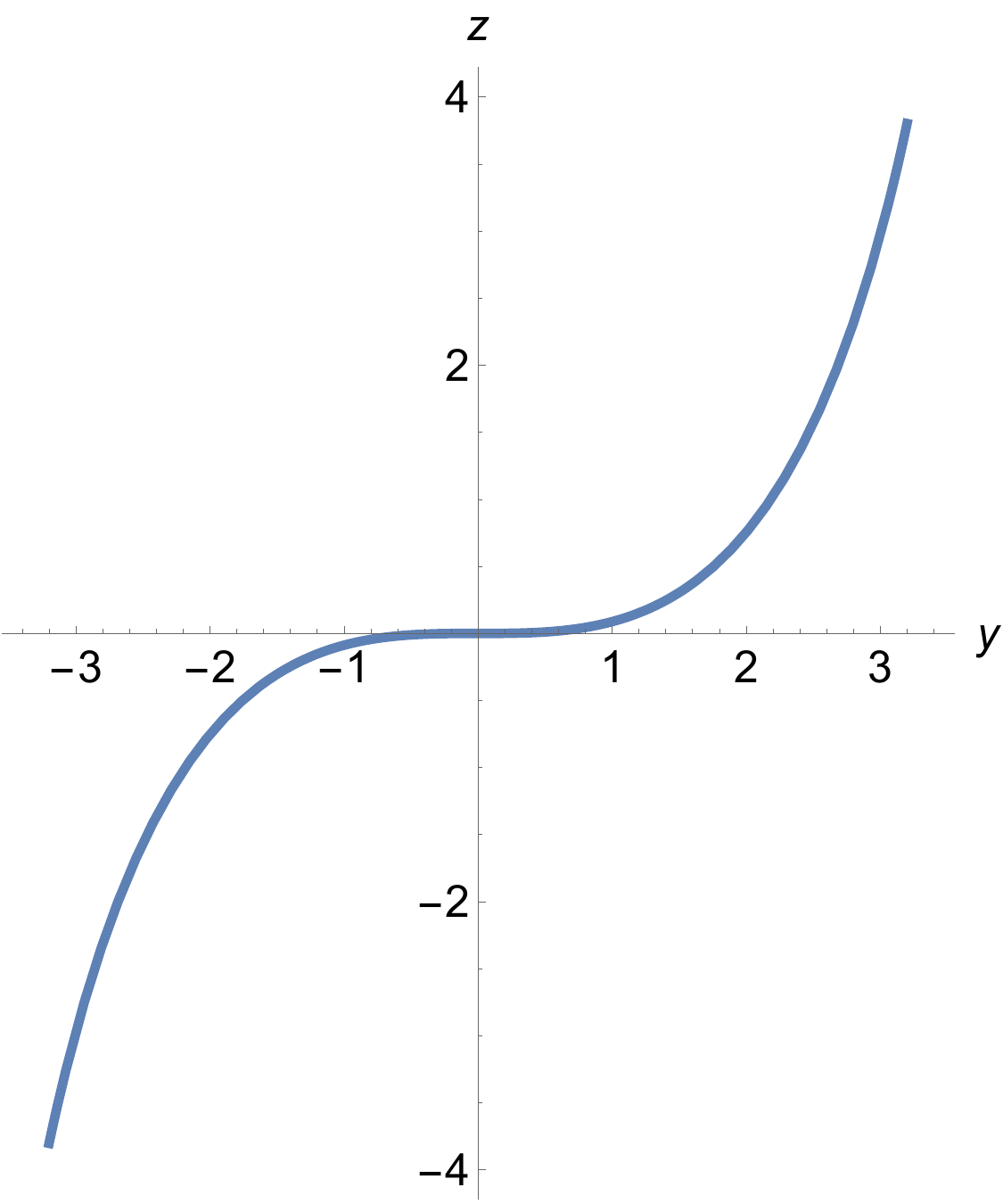} 
\caption{Left: horizontal invariant$F_1$-soliton governed by the ODE \eqref{16}.   Right:  horizontal invariant $F_3$-soliton governed by the ODE \eqref{17}.  }
\label{fig1}
\end{figure}

\section{Invariant solitons with respect to the vector field $F_4=-y\partial_x+x\partial_y$}\label{s6}

In the last section, we classify the invariant solitons corresponding to the  Killing vector field $F_4 = -y\partial_x + x\partial_y$. This vector 
field imposes restrictions to the existence of invariant solitons, such as ruling out  horizontal translation solitons.  In the case of rotational solitons,   they must be flat  ($K=0$), which is proved in the following result.
 
\begin{proposition}\label{pr61}
An invariant surface $\Sigma$ in $\text{Nil}_3$ is an $F_4$-soliton if and only if one of the following conditions holds:
\begin{enumerate}
  \item $\Sigma$ is invariant under vertical translations, in which case, $\Sigma$ is parametrized by  
  \begin{equation}\label{f4v}
  \Psi(s,t)=\sqrt{C - \frac{1}{2}s}\, (\cos\theta(s),\sin\theta(s),t),
  \end{equation}
  where 
  $$  \theta(s)= \mp \left( \sqrt{16C - 8s - 1} - \arctan\left(\sqrt{16C - 8s - 1}\right) \right) + \theta_0,$$
 with $C, \theta_0 \in \mathbb{R}$. The solution is defined on   $s \le 2C - 1/8$. In particular,   the squared Euclidean distance of $\Sigma$ to the $z$-axis decreases linearly as it approaches the circle $R = 1/4$: see Fig. \ref{fig2}, left.

    \item There are no $F_4$-solitons invariant under horizontal translations.
    
    \item $\Sigma$ is invariant under helicoidal motions, in which case,  its generating curve $\gamma(s)=(r(s),0,h(s))$, with $r'^2+h'^2=1$, satisfies the ODE
    \begin{equation}\label{f4r0}
        \begin{split}
        & \left( 1 - c + \frac{r^2}{2} \right)\left( -r^3r'' + r^2r'^2(1 - r'^2) ( c - \frac{r^2}{2}  ) \right) \\
        &\quad - \left[ r'^2 \left( c - \frac{r^2}{2} \right) \left( \frac{c}{2} - \frac{r^2}{4} - 1 \right) - \frac{r^2}{2} \right]^2   = c r r' \sqrt{W} ,
        \end{split}
    \end{equation}
    where $W =  r^2 + r'^2\left(c-\frac{r^2}{2}\right)^2$. If the surface is   rotational ($c=0$), then $\Sigma$ is flat ($K=0$) and characterized by  
    \begin{equation}\label{f4r}
     \left( 1 + \frac{r^2}{2} \right)\left(r^3r'' +\frac{1}{2}r^4r'^2(1 - r'^2) \right) - \frac{r^4}{4}\left[ r'^2 \left( 1 + \frac{r^2}{4} \right) - 1 \right]^2 = 0.
    \end{equation}  
    \end{enumerate}
\end{proposition}

\begin{proof}
By writing  $F_4 $ in terms of  $\mathcal{B}$,  the inner product $\langle N,F_4\rangle$ is  given by 
\begin{equation}\label{F4}
    \langle N, F_4 \rangle = -y N_1 + x N_2 - \frac{x^2+y^2}{2} N_3.
\end{equation}

\begin{enumerate}
    \item Assume $\Sigma$  parametrized by \eqref{p1}, with $x'^2+y'^2=1$. From \eqref{n1} and  \eqref{F4}, we find $\langle N, F_4 \rangle = -y y' - x x'$. Since   $K = -1/4$ by \eqref{k1}, the soliton equation \eqref{eq1} becomes $x x' + y y' = -1/4$.     Integrating, we find     $x(s)^2 + y(s)^2 = -\frac{1}{2}s + C$, where $C$ is a constant. Furthermore, applying the Cauchy-Schwarz inequality to the vectors $(x, y)$ and $(x', y')$, and taking into account the arc-length condition $x'^2 + y'^2 = 1$, we obtain
    \begin{equation*}
        \frac{1}{16} = (x x' + y y')^2 \le \left(x^2 + y^2\right)\left(x'^2 + y'^2\right) = x^2 + y^2=-\frac12s+C.
    \end{equation*}
 This  establishes the maximal domain for the parameter $s$.   If we express $\gamma$ in polar coordinates $(\rho(s),\theta(s))$,  then we find  $\rho(s) = \sqrt{C - \frac{1}{2}s}$.   The arc-length condition $x'^2 + y'^2 = 1$ is written  in polar coordinates as $\rho'^2 + \rho^2 \theta'^2 = 1$.  This implies 
$$\theta'(s) = \pm \frac{\sqrt{1-\rho'^2}}{\rho}=\pm \frac{\sqrt{16 C-8 s-1}}{2 (s-2 C)}.$$
Integrating, we obtain the expression for $\theta$ given in   item (1).   
    
    \item For horizontal translation-invariant surfaces parametrized by \eqref{p2}, with $y'^2+z'^2=1$, and thanks to \eqref{n2}, equation \eqref{F4} becomes
        \begin{equation*}
        \langle N, F_4 \rangle =  \frac{1}{\sqrt{W}}\left( \frac{y^2 y'}{2} - t z' - \frac{t^2}{2}y' \right).
    \end{equation*}
  Since the left-hand side of the soliton equation \eqref{eq1} does not depend on $t$, we deduce   $z'=0$ and $y'=0$. This contradicts 
   the arc-length condition $y'^2 + z'^2 = 1$. Thus, no such solitons exist.
    
    \item For helicoidal surfaces parametrized by \eqref{p3}, with $r'^2+h'^2=1$, the right-hand side of \eqref{eq1} is obtained by using \eqref{n3} and  \eqref{F4}, yielding 
    \begin{equation*}
        \langle N, F_4 \rangle =   -\frac{c r r'}{\sqrt{W}}.
    \end{equation*}
  Utilizing   \eqref{k3}, and simplifying via $h'^2=1-r'^2$ and $h'h''=-r'r''$,   we arrive   at the differential equation \eqref{f4r0}. If $c=0$, then $\langle N, F_4 \rangle = 0$, which   implies   $K=0$ because of \eqref{eq1}.
\end{enumerate}
\end{proof}

The equation in item (3) of Proposition \ref{pr61} is difficult to study. However, the case of rotational solitons ($c=0$) is more manageable because  $K=0$. The following result gives the   geometric properties of these surfaces, demonstrating the existence of rotational solitons that intersects the generating axis at two cusps, or orthogonally, or not intersect the rotation axis at all (see Fig. \ref{fig2}, right).

\begin{theorem} 
Let $\Sigma_{r_0}$ denote a rotational $F_4$-soliton whose profile curve is   $\gamma_{r_0}(s) = (r(s),0, h(s))$ parameterized by arc-length and satisfying the initial conditions $r(0) = r_0>0$, $h(0)=0$ and $\theta(0) = \pi/2$.  Then $\gamma$ is a graph over a bounded interval of the $z$-axis. Moreover, there exists a critical radius $r_c > 0$ such that the   the family $\Sigma_{r_0}$ exhibits the following properties: 
\begin{enumerate}
    \item If $r_0 < r_c$, then  $\gamma_{r_0}$  intersects the rotation axis with   two conical singularities (or cusps) at the poles.
        \item If $r_0 =r_c$, then  $\gamma_{r_0}$  intersects the rotation axis orthogonally and the corresponding rotational surface is a closed surface which is $C^1$ at the poles.  
\item If $r_0 > r_c$, then  $\gamma_{r_0}$ does not intersect the rotation axis and    reaches   horizontal tangent vectors at the endpoints of $\gamma$.  
\end{enumerate}
\end{theorem}

\begin{proof}
The arc-length condition $r'^2+h'^2=1$ implies the existence of a smooth function $\theta$ such that 
 $r' = \cos\theta$ and $h' = \sin\theta$.   Then, together \eqref{f4r}, the generating curve $\gamma$ is solution of the system 
 \begin{equation}\label{s33}
 \left\{
 \begin{split}
 r'&=\cos\theta,\\
 h'&=\sin\theta,\\
 \theta'&=  \frac{G(r,\theta)}{\sin\theta},
 \end{split}\right.
 \end{equation}
where
\begin{equation*}
G(r, \theta) = \frac{r}{2}\cos^2\theta\sin^2\theta + \frac{r}{4\left(1+\frac{r^2}{2}\right)}\left( \frac{r^2}{4}\cos^2\theta - \sin^2\theta \right)^2.
\end{equation*}
Since the function $z$ does not appear in equation \eqref{f4r}, a vertical translation in $\nil$ does not change the geometry of the surface. Thus, in the system \eqref{s33} we can choose as initial conditions,  $h(0)=0$ and $\theta(0)=\pi/2$  without loss   
of generality.  This means that the initial hypothesis in the theorem on the initial conditions covers all possible situations.  An argument similar to the one in the proof of Theorem \ref{t32} proves that $\gamma$ is symmetric about the horizontal line of equation $z=0$ in the $yz$-plane.

The constraint $r'^2\leq 1$ and because $\sin\theta\not=0$ implies that $\theta\in (0,\pi)$. Thus $r'\not=0$, proving tha $\gamma$ is a graph on the $z$-axis and writing $r=r(z)$, with $z=h$. Differentiating $r'(s)$, we have $r'' = -\theta'\sin\theta = -G(r,\theta)<0$   for all $r>0$ and $\theta \in (0, \pi)$. It follows that $r''(s) < 0$. Thus,  the initial radius $r_0$ is indeed the strict global maximum 
of the curve. In addition, $r''(s)=-\theta'/\sin^3\theta<0$, which proves that $\gamma$ is a global concave graph on the $z$-axis.   The   maximal domain of solutions of \eqref{s33} is determined  by when the function $r(s)$ attains the value $0$ or the function $\theta(s)$   attains the values $0$ or $\pi$.  

The function $G$ has the symmetry property $G(r, \pi-\theta) = G(r, \theta)$. Moreover, the derivative  of $r$ with respect to $\theta$ is
\begin{equation}\label{32}
\frac{dr}{d\theta} =\frac{r'}{\theta'}= \frac{\cos\theta\sin\theta}{G(r, \theta)}.
\end{equation}
Due to the symmetry of $G$, we focus on  $\theta \in (0, \pi/2]$. Integrating \eqref{32},  we obtain
\begin{equation}\label{33}
  r(\theta) -r_0= \int^{\theta}_{\pi/2} \frac{\cos\theta\sin\theta}{G(r(\theta), \theta)} \, d\theta.
\end{equation}

We now study this integral in terms of   $r_0$:
\begin{enumerate}
\item  We prove that $\gamma$ does not reach $r=0$ if $r_0$ is sufficiently large. The change of variables $u = \cos^2\theta$ in \eqref{33} yields 
\begin{equation}\label{35}
r_0 - r(\theta) = \frac{1}{2}\int_0^u \frac{1}{D(r(u), u)} \, du,
\end{equation}
where $u\in [0,1]$ and the denominator is given by
\begin{equation}\label{dd}
D(r, u) = \frac{r}{2}u(1-u) + \frac{r}{4\left(1+\frac{r^2}{2}\right)}\left( \left(\frac{r^2}{4}+1\right)u - 1 \right)^2.
\end{equation}

To   evaluate the integral \eqref{35}, we proceed   to show that   $r(u) \ge r_0/2$ holds for all $u\in [0,1]$. The argument is by contradiction. Suppose that this claim is false. Since $r(\pi/2)=r_0$, by continuity, there must exist a first point $u^*\in (0,1)$ where $r(u^*)=r_0/2$, meaning $r(u)\geq r_0/2$ holds for all $u\in[0,u^*]$. Within this interval,   we bound $D(r,u)$ from below by analyzing two subintervals for $u$.
\begin{enumerate}
\item Assume $u \in [0, 8/r^2]$. If $r$ is large, we have 
$$ 
 \frac{r}{4(1+r^2/2)}  = \frac{1}{2r} + \mathcal{O}(r^{-3}).$$
 Using    $u \le 8/r^2$, we have 
$$ \left(\frac{r^2}{4}+1\right)u  = \frac{r^2}{4}u + \mathcal{O}(r^{-2}) \le 2 + \mathcal{O}(r^{-2}).$$
Since we are assuming $r(u) \ge r_0/2$ on the interval $[0, u^*]$, the constants in the $\mathcal{O}(r^{-2})$ and 
$\mathcal{O}(r^{-3})$ terms are uniformly bounded independent of $u$ and $r_0$ (for sufficiently large $r_0$).  
Multiplying these asymptotic expressions and factoring out $1/r$ yields, we arrive at  
\begin{equation*}
D(r,u) = \frac{1}{r} \left[ \frac{r^2 u}{2}(1-u) + \frac{1}{2}\left(\frac{r^2 u}{4} - 1\right)^2 \right] + \mathcal{O}(r^{-3}).
\end{equation*}
Since the term in the bracket is a positive continuous function of the bounded quantity $r^2 u \in [0,8]$, it is uniformly bounded from below by a constant, yielding $D(r,u) \ge \frac{c_1}{r}$, with $c_1\in\r$.

\item Assume $u \in [8/r^2, 1]$. In this interval, the first term of $D(r,u)$ in \eqref{dd} is non-negative. Moreover, since $u \ge 8/r^2$, we have 
\begin{equation*}
\left( \left(\frac{r^2}{4}+1\right)u - 1 \right)^2 > \left( \frac{r^2}{4}u - 1 \right)^2 \ge \left( \frac{r^2}{8}u \right)^2 = \frac{r^4}{64}u^2.
\end{equation*}
For large $r$, we have $\frac{r}{4(1+r^2/2)} \ge \frac{c}{r}$ for some uniform constant $c > 0$. Dropping the first non-negative term of $D$, we obtain $ D(r,u) \ge  c_2 r^3 u^2$, for some positive constant $c_2$.
\end{enumerate}

Now, applying these bounds by integrating \eqref{35}   from $0$ to $u^*$, and noting that the value $8/r(u)^2$ lies   between $8/r_0^2$ and $32/r_0^2$, we have:
\begin{equation*}
\begin{split}
r_0 - r(u^*) &\le \frac{1}{2} \int_0^{32/r_0^2} \frac{r(u)}{c_1} \, du + \frac{1}{2} \int_{8/r_0^2}^1 \frac{1}{c_2 r(u)^3 u^2} \, du \\
&\le \frac{1}{2} \int_0^{32/r_0^2} \frac{r_0}{c_1} \, du + \frac{1}{2} \int_{8/r_0^2}^1 \frac{8}{c_2 r_0^3 u^2} \, du \\
&\le \frac{16}{c_1 r_0} + \frac{1}{2 c_2 r_0} := \frac{c_3}{r_0},
\end{split}
\end{equation*}
for some constant $c_3 > 0$. For a sufficiently large $r_0$, we  have $c_3/r_0 < r_0/2$, which   implies that $r(u^*)\geq r_0-c_3/r_0>r_0/2$. This contradicts our assumption that $r(u^*)=r_0/2$. Consequently, no such point $u^*$ can exist, meaning that the bound holds for all $u\in [0,1]$ and therefore, 
$$r(0)=r(\theta=0)\geq r_0-\frac{c_3}{r_0}>0.$$
This proves that $\gamma$ is bounded away from the $z$-axis if $r_0$ is sufficiently large.  

At   $\theta = 0$, we have $r'(0)=1$, which shows that $\gamma$ ends with a horizontal tangent vector. The argument in the interval $[\pi/2,\pi)$ is a consequence of the symmetry of $\gamma$ with respect to the line $z=0$.

\item Now, let $r\to 0$. A standard Taylor expansion yields $\left(1+\frac{r^2}{2}\right)^{-1} = 1 + \mathcal{O}(r^2)$, which   implies
$$G(r, \theta) = \frac{r}{2}\cos^2\theta\sin^2\theta + \frac{r}{4}\sin^4\theta + \mathcal{O}(r^3) = \frac{r\sin^2\theta}{4}(1+\cos^2\theta) + \mathcal{O}(r^3).$$
Consequently, for $r$ sufficiently small, we can bound $G(r,\theta)$, and equation \eqref{32} yields the differential inequality
\begin{equation*}
\frac{dr}{d\theta} \geq \frac{c_4\cos\theta}{r\sin\theta(1+\cos^2\theta)},
\end{equation*}
for some uniform constant $0 < c_4 < 2$. Separating variables and integrating   from $\theta=\pi/2$ yields 
\begin{equation*}
r(\theta)^2 \leq r_0^2 + c_4\log \left( \frac{\sin^2\theta}{1+\cos^2\theta} \right).
\end{equation*}
As $\theta\to 0$, since the logarithmic term diverges to $-\infty$,   the upper bound for $r(\theta)^2$ in this inequality becomes negative. Therefore,  there must exist a minimum   angle $\theta_{\min} > 0$ such that $r(\theta_{\min})=0$.  This implies that   $\gamma$ attains the $z$-axis before   reaching a horizontal tangent vector. Moreover, since $\theta_{min}\not=0$, $r'(\theta_{min})\not=1$, which means that the intersection of $\gamma$ with the $z$-axis is not orthogonal, that is, it is a cusp. 

The existence of the   critical  value $r_c$ follows  from the continuous dependence of solutions on the 
initial parameter $r_0$. For the value $r_0=r_c$, it  must satisfy $r \to 0$ and $\theta \to 0$ simultaneously, which means the profile curve meets the vertical rotation axis, and at this intersection, $\gamma$ is orthogonal to the $z$-axis.

\end{enumerate}
\end{proof}
 
  \begin{figure}[h!t]
\centering
\includegraphics[width=.35\linewidth]{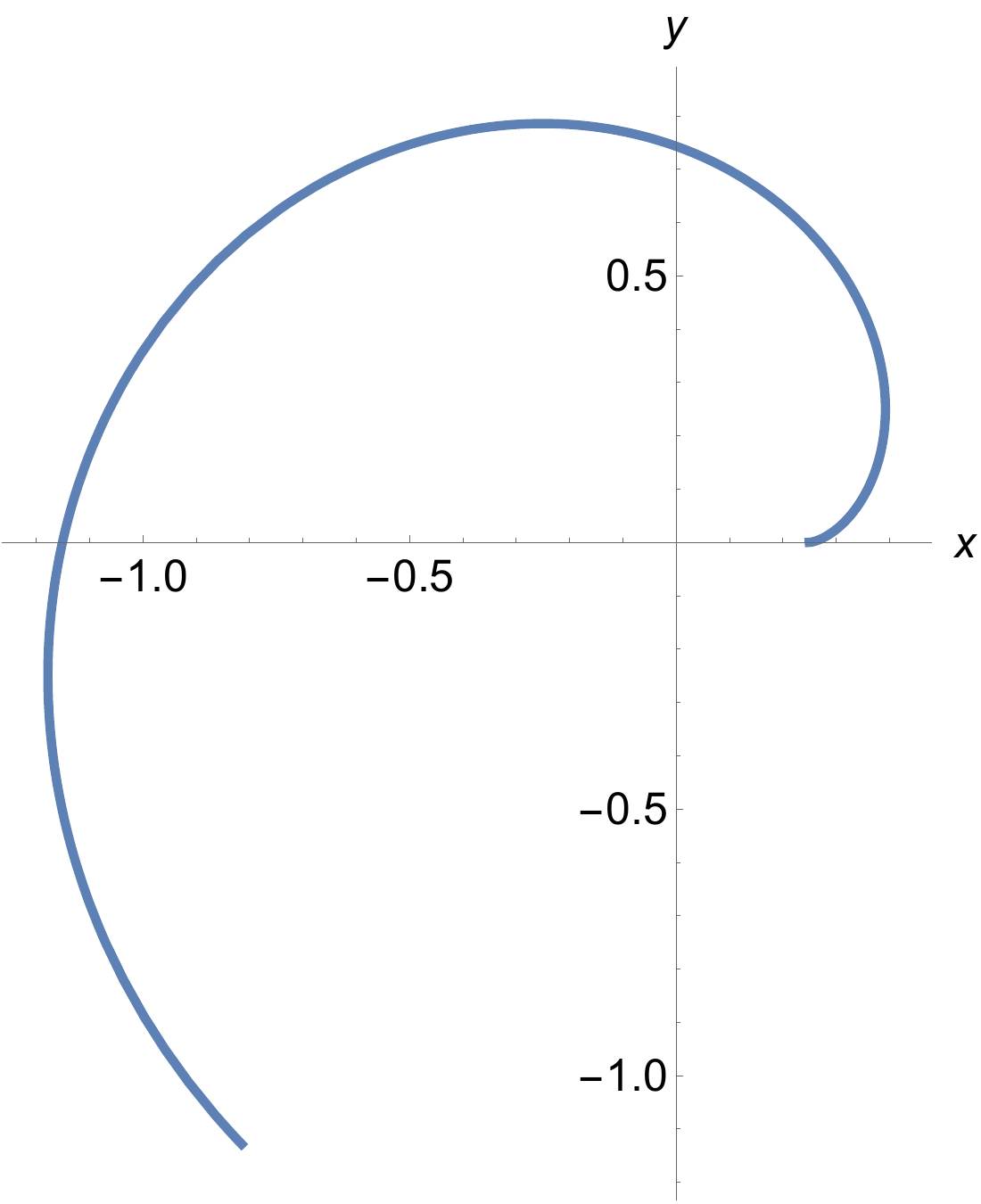}\qquad \qquad\qquad\includegraphics[width=0.25\linewidth]{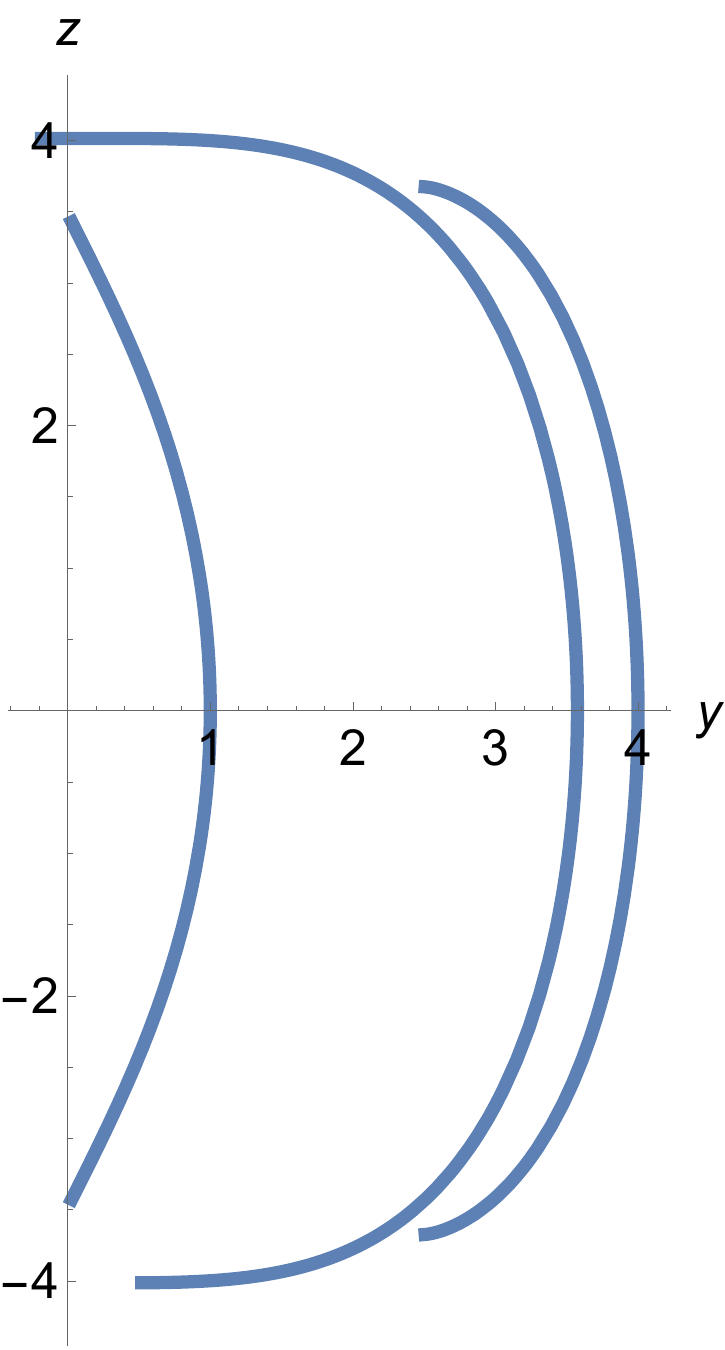} 
\caption{Left: vertical invariant $F_4$-soliton parametrized by    \eqref{f4v}.   Right:  rotational $F_4$-solitons governed by the ODE \eqref{f4r}, intersecting and not intersecting the rotation axis.  }
\label{fig2}
\end{figure}

 \section*{Acknowledgements}
Rafael Belli thanks the Department of Geometry and Topology of the University of Granada for their hospitality, where this work was carried out. This study was financed, in part, by the São Paulo Research Foundation (FAPESP), Brasil.
 Process Number 2025/09871-1. Rafael L\'opez has been partially supported
  by MINECO/MICINN/ FEDER grant
no. PID2023-150727NB-I00, and by the ``Mar\'{\i}a de Maeztu'' Excellence
Unit IMAG, reference CEX2020-001105- M, funded by MCINN/AEI/10.13039/
501100011033/ CEX2020-001105-M.

\section*{Ethics declarations}

Conflict of interest. The authors declare that they have no conflict of interest. No datasets were generated or analysed during the current study.



\end{document}